\documentclass[11pt]{amsart}
\usepackage{indentfirst,enumerate,cite,amssymb,amsfonts,amsmath,amsthm,mathrsfs,dsfont,mathtools}
\usepackage{color}
\usepackage{graphicx}
\usepackage{multicol}
\usepackage[colorlinks=true,linkcolor=blue,citecolor=blue]{hyperref}
\usepackage{enumerate}
\usepackage{geometry}
\numberwithin{equation}{section}

\newtheorem{theorem}{Theorem}[section]
\newtheorem{definition}[theorem]{Definition}

\newtheorem{lemma}[theorem]{Lemma}
\newtheorem{proposition}[theorem]{Proposition}
\newtheorem{corollary}[theorem]{Corollary}

\newtheorem*{remark}{Remark}

\allowdisplaybreaks

\begin{document}
	
	\title[Global Hypoellipticity for Involutive Systems]{Global Hypoellipticity for Involutive Systems  on Non-Compact Manifolds}

	\author[S. Coriasco]{Sandro Coriasco}
	\address{
		Dipartimento di Matematica ``Giuseppe Peano'', 
		Università Degli Studi di Torino, 
		Via Carlo Alberto 10, CAP 0123, Torino, 
		Italia }
	\email{sandro.coriasco@unito.it}

	\author[A. Kirilov]{Alexandre Kirilov}
	\address{
		Departamento de Matem\'atica, 
		Universidade Federal do Paran\'a,  
		Caixa Postal 19096, CEP 81530-090, Curitiba, Paran\'a, 
		Brasil}
	\email{akirilov@ufpr.br}
	
	\author[W. de Moraes]{Wagner A. A. de Moraes}
	\address{
		Departamento de Matem\'atica,  
		Universidade Federal do Paran\'a,  
		Caixa Postal 19096\\  CEP 81530-090, Curitiba, Paran\'a, 
		Brasil}
	\email{wagnermoraes@ufpr.br}
	
	\author[P. Tokoro]{Pedro M. Tokoro}
	\address{
		Programa de P\'os-Gradua\c c\~ao em Matem\'atica,  
		Universidade Federal do Paran\'a,  
		Caixa Postal 19096\\  CEP 81530-090, Curitiba, Paran\'a,  
		Brasil}
	\email{pedro.tokoro@ufpr.br}
	\thanks{This study was financed in part by the Coordenação de Aperfeiçoamento de Pessoal de Nível Superior - Brasil (CAPES) - Finance Code 001. The first author was also partially supported by 
		the Italian Ministry of the University and Research - MUR, within the framework of the Call relating to the scrolling of the final rankings of the 
		PRIN 2022 - Project Code 2022HCLAZ8, CUP D53C24003370006 (PI A. Palmieri, Local unit Sc. Resp. S. Coriasco). The second author was supported in part by CNPq - Brasil (grant 316850/2021-7).}

	

	\begin{abstract}
		
			We study the global hypoellipticity of the operator 
			\(
			\mathbb{L} = \mathrm{d}_t + \sum_{k=1}^m \omega_k \wedge \partial_{x_k},
			\)
			defined on differential forms over product manifolds of the form $M \times \mathbb{T}^m$, where $M$ is a non-compact manifold homeomorphic to the interior of a compact manifold with boundary, equipped with a scattering metric, and $\omega_1,\dots,\omega_m$ are smooth closed 1-forms on $M$. Extending previous results obtained in the compact setting, we characterize global hypoellipticity of $\mathbb{L}$ in terms of arithmetic properties of the forms $\omega_k$. The analysis relies on microlocal techniques adapted to the scattering setting and a version of the Hodge Theorem for scattering manifolds.
		
	\end{abstract}

	\maketitle

\section{Introduction}

In this paper, we study the global hypoellipticity of a class of differential operators acting on differential forms over product manifolds of the form $M \times \mathbb{T}^m$, where $M$ is a non-compact Riemannian manifold which is homeomorphic to the interior of a compact manifold with boundary, and $\mathbb{T}^m$ is the $m$-dimensional torus. 

Scattering manifolds, introduced by Melrose \cite{Melrose_GST}, are smooth manifolds with boundary whose geometry at infinity is modeled by asymptotically conic or Euclidean structures. The metric $g$ on the interior of a compact manifold with boundary is said to be a scattering metric if, near the boundary, it takes the form
\[
g = \frac{ \mathrm{d}x^2}{x^4} + \frac{g'}{x^2},
\]
where $x$ is a boundary defining function and $g'$ is a symmetric 2-tensor restricting to a metric on the boundary. These metrics define a natural class of non-compact manifolds with well-behaved asymptotic properties, allowing for the development of analysis at infinity, including versions of the Hodge theorem and Hodge cohomology \cite{Melrose_APS,Melrose_GST}.

Let us recall the notion of global hypoellipticity for differential operators acting on sections of vector bundles over smooth manifolds:

\begin{definition}\label{GH}
	Let $E, F$ be vector bundles over a smooth manifold $M$. We say that a linear operator $P: C^\infty(M,E) \to C^\infty(M,F)$ 
	is \emph{globally hypoelliptic} if
	\[
	u \in \mathscr{D}'(M,E), \quad Pu \in C^\infty(M,F) \quad \Rightarrow \quad u \in C^\infty(M,E).
	\]
\end{definition}

Our main object of study is the first-order operator
\[
\mathbb{L}: C^\infty(M \times \mathbb{T}^m) \longrightarrow \mathsf{\Lambda}^1 C^\infty(M \times \mathbb{T}^m), \quad \mathbb{L}u = \mathrm{d}_t u + \sum_{k=1}^{m} \omega_k \wedge \partial_{x_k} u,
\]
where $\mathrm{d}_t$ denotes the exterior derivative acting on the $M$-variable, $M$ belongs to a suitable class of non-compact manifolds (see below), $\omega_1, \dots, \omega_m$ are smooth real-valued closed 1-forms on $M$, and $\mathbb{T}^m \simeq \mathbb{R}^m / 2\pi \mathbb{Z}^m$ is the $m$-dimensional torus with angular coordinates $x = (x_1, \dots, x_m)$.

Operators of this form arise naturally in the context of involutive structures, where the associated distribution is closed under the Lie bracket. Locally, they can be interpreted as systems of commuting first-order partial differential equations, with a geometric structure encoded by the 1-forms $\omega_k$. For a general framework on involutive systems and their analytic properties, we refer the reader to \cite{BCH_book,Treves}.

The global hypoellipticity of operators of this form has been previously studied in the compact setting. When $M$ is a closed manifold and $m=1$, Bergamasco, Cordaro, and Malagutti \cite{BCM1993} characterized the global hypoellipticity of $\mathbb{L}$ in terms of arithmetic properties of the 1-form $\omega$, namely whether it is integral, rational, or Liouville. Their analysis relies on classical Hodge theory on compact manifolds, especially the identification of harmonic representatives for cohomology classes. This approach was extended in \cite{ADL2023gh} to the case $m \geq 1$, through the introduction of appropriate generalizations of rational and Liouville families of closed 1-forms.

Building on these results, we provide a characterization of the global hypoellipticity of $\mathbb{L}$ when $M$ is a scattering manifold. Our strategy combines techniques from microlocal analysis on manifolds with boundary with scattering geometry and cohomology theory adapted to the non-compact setting. In particular, we make use of a version of Hodge theory developed for scattering metrics, which allows us to define harmonic representatives of cohomology classes in a suitable weighted setting.

\section{Geometric and Analytic Background}

This section provides the geometric and analytic background required for the development of our results. We begin by recalling the notion of asymptotically Euclidean manifolds and review some basic properties of elliptic complexes. We then introduce the concept of relative de Rham cohomology and specialize it to the context of asymptotically Euclidean  manifolds, where the boundary plays a central role. In the following subsection, we describe the version of Hodge theory that is suited to this class of non-compact manifolds. Finally, we present the construction of partial Fourier series on the product of an asymptotically Euclidean manifold with an $n$-dimensional torus, a tool that will be fundamental to our regularity analysis.

\subsection{Asymptotically Euclidean Manifolds} \

Let $\overline{M}$ be a smooth manifold with boundary, whose interior we denote by $M$. A \textit{boundary defining function} is a smooth map $x: \overline{M} \to [0, +\infty)$ such that $x^{-1}(0) = \partial M$ and $\mathrm{d}x_p \neq 0$ for all $p \in \partial M$. Consider $\mathbb{R}^n$ equipped with the standard Riemannian metric. Through radial compactification, we may identify $\mathbb{R}^n$ with the interior of the closed unit ball $\mathbb{B}^n$, where the boundary $\partial \mathbb{B}^n = \mathbb{S}^{n-1}$ corresponds to points at infinity. Specifically, let $\rho: \mathbb{R}^n \to \mathbb{B}^n$ be a diffeomorphism defined in polar coordinates by
\[
\rho(r, \theta) = \left(1 - \frac{1}{r}, \theta\right),
\]
outside a sufficiently small neighborhood of the origin. Let $[\cdot]: \mathbb{R}^n \to [0, +\infty)$ be a smooth function such that $[y] = |y|$ for $|y| \geq 3$. Then, the map $x: \mathbb{B}^n \to [0, +\infty)$ defined by
\[
x(\eta) = \frac{1}{[\rho^{-1}(\eta)]}
\]
is a boundary defining function. Moreover, the Riemannian metric on $\mathbb{R}^n$ induces a metric on $\mathbb{B}^n$ via radial compactification, which near the boundary takes the form
\begin{equation}\label{scattering-metric}
	g = \frac{\mathrm{d}x^2}{x^4} + \frac{g'}{x^2},
\end{equation}
where $g'$ denotes the lift of the standard metric on $\mathbb{S}^{n-1}$.

More generally, an \textit{asymptotically Euclidean manifold} is a smooth compact manifold with boundary $\overline{M}$, whose interior $M$ is endowed with a Riemannian metric $g$ that, near the boundary, assumes the form \eqref{scattering-metric}, where $x$ is a boundary defining function and $g'$ is a smooth symmetric 2-tensor field on $\overline{M}$ that restricts to a metric on $\partial M$. In this setting, we say that $g$ is a \textit{scattering metric}. For this reason, asymptotically Euclidean manifolds are also referred to as \textit{scattering manifolds}. For further details, we refer the reader to \cite{cordes, CD2021, Melrose_GST}.

A well-known subclass of scattering manifolds is formed by manifolds with cylindrical ends. Roughly speaking, a manifold with one cylindrical end is a non-compact manifold $M$ constructed as follows: starting from a smooth compact manifold without boundary $M_0$ of dimension $n$, we remove an open set whose boundary is diffeomorphic to the sphere $\mathbb{S}^{n-1}$. In a collar neighborhood of this removed region, we smoothly attach a cylindrical end diffeomorphic to $\mathbb{S}^{n-1} \times \mathbb{R}$. This construction can be adapted to produce manifolds with finitely many cylindrical ends.

Alternatively, $M$ can be identified with the interior of a compact manifold with boundary $\overline{M}$, where $\partial M$ is diffeomorphic to a finite disjoint union of spheres. The interior of this manifold can be equipped with a Riemannian metric $g$ satisfying appropriate conditions, thereby placing $M$ within the class of scattering manifolds. This identification is made precise via radial compactification, as discussed in \cite{CD2021}, and allows for the construction of a scattering metric on the interior of $\overline{M}$, as shown in \cite{Melrose_GST}.

\subsection{Elliptic Differential Complexes}\label{sect_ellcomp} \

Let $E^0, \dots, E^N$ be complex vector bundles over a Riemannian manifold $M$, and consider a sequence of first-order differential operators $P_q: C^\infty(E^q) \to C^\infty(E^{q+1})$ for $q = 0, \dots, N-1$, satisfying the condition:
\[
P_{q+1} \circ P_q = 0, \quad q = 0, \dots, N-2.
\] 

We say that the complex
\begin{equation}\label{elliptic_complex}
	0 \to C^\infty(E^0) \overset{P_0}{\longrightarrow} C^\infty(E^1) \overset{P_1}{\longrightarrow} \cdots \overset{P_{N-1}}{\longrightarrow} C^\infty(E^N) \to 0
\end{equation}
is \textit{elliptic} if, for every $t \in M$ and every $\tau \in T^*_t M \setminus \{0\}$, the sequence of finite-dimensional vector spaces and linear maps
\[
0 \to E^0_t \overset{\sigma_t(P_0)(\tau)}{\longrightarrow} E^1_t \overset{\sigma_t(P_1)(\tau)}{\longrightarrow} \cdots \overset{\sigma_t(P_{N-1})(\tau)}{\longrightarrow} E^N_t \to 0
\]
is exact, where $\sigma_t(P_q)(\tau)$ denotes the principal symbol of $P_q$ at $(t, \tau)$.

To define the formal adjoint of each operator $P_q$, we equip each vector bundle $E^q$ with a Hermitian metric $\langle \cdot, \cdot \rangle_{E^q}$, and let $\mathrm{d}V$ be the volume form induced by the Riemannian metric on $M$. For each $q = 0, \dots, N-1$, the formal adjoint $P_q^*: C^\infty(E^{q+1}) \to C^\infty(E^q)$ is defined by the relation
\[
\int_M \langle P_q u, v \rangle_{E^{q+1}} \, \mathrm{d}V = \int_M \langle u, P_q^* v \rangle_{E^q} \, \mathrm{d}V, \quad \text{for all } u \in C^\infty(E^q), \; v \in C_0^\infty(E^{q+1}).
\]

This gives rise to the adjoint complex:
\[
0 \leftarrow C^\infty(E^0) \overset{P_0^*}{\longleftarrow} C^\infty(E^1) \overset{P_1^*}{\longleftarrow} \cdots \overset{P_{N-1}^*}{\longleftarrow} C^\infty(E^N) \leftarrow 0.
\]

Associated with the complex \eqref{elliptic_complex}, we define the Laplacian at each level $q = 0, \dots, N$ by
\[
\mathcal{L}_q := P_q^* P_q + P_{q-1} P_{q-1}^*: C^\infty(E^q) \to C^\infty(E^q),
\]
where we set $P_{-1} = P_N = 0$ by convention. Each $\mathcal{L}_q$ is elliptic whenever the complex \eqref{elliptic_complex} is elliptic. For further details see \cite[Section 1.5]{Gilkey}.

\begin{proposition}\label{gh_firstlevel}
	If the complex \eqref{elliptic_complex} is elliptic, then the operator $P_0$ is globally hypoelliptic.
\end{proposition}
\begin{proof}
	Let $u \in \mathscr{D}'(E^0)$ be a distributional section such that $P_0 u \in C^\infty(E^1)$. Since for $q = 0$ we have $\mathcal{L}_0 = P_0^* P_0$, it follows that
	\[
	\mathcal{L}_0 u = P_0^* P_0 u \in C^\infty(E^0).
	\]
	By ellipticity of $\mathcal{L}_0$, we conclude that $u \in C^\infty(E^0)$.
\end{proof}

\subsection{Relative Cohomology on Asymptotically Euclidean Manifolds}\label{ssec_rel_coh} \

	Let $M$ be a smooth manifold, $S \subset M$ a submanifold, and $i: S \to M$ the inclusion map. Define the complex $(A^\bullet(M, S), \mathrm{d})$ by setting
	\[
	A^k(M, S) = \mathsf{\Lambda}^k C^\infty(M) \oplus \mathsf{\Lambda}^{k-1} C^\infty(S),
	\]
	with differential
	\[
	\mathrm{d}(\omega, \theta) = (\mathrm{d}\omega, i^*\omega - \mathrm{d}\theta).
	\]
	It is easy to verify that $\mathrm{d}^2 = 0$. Note that exact elements in this complex correspond to closed forms on $M$ whose pullbacks to $S$ are exact.
	
	Consider the maps
	\[
	\alpha: \theta \in \mathsf{\Lambda}^{k-1} C^\infty(S) \mapsto (0, \theta) \in A^k(M, S),
	\]
	and
	\[
	\beta: (\omega, \theta) \in A^k(M, S) \mapsto \omega \in \mathsf{\Lambda}^k C^\infty(M).
	\]
	Then, for each $k$, we have a short exact sequence
	\[
	0 \to \mathsf{\Lambda}^{k-1} C^\infty(S) \xrightarrow{\alpha} A^k(M, S) \xrightarrow{\beta} \mathsf{\Lambda}^k C^\infty(M) \to 0,
	\]
	which induces the long exact sequence in cohomology:
	\[
	\cdots \to H^k(M, S) \xrightarrow{\beta^*} H^k(M) \xrightarrow{i^*} H^k(S) \xrightarrow{\alpha^*} H^{k+1}(M, S) \to \cdots.
	\]

	For further details, we refer the reader to the classical treatment by Bott and Tu \cite{bott_tu}.

	Now suppose that $\overline{M}$ is a scattering manifold and $S = \partial M$. From the exactness of the sequence, we obtain the decompositions
	\[
	H^k(\overline{M}) \simeq \operatorname{Im}(\beta^*) \oplus \operatorname{Im}(i^*) \quad \text{and} \quad H^k(\overline{M}, \partial M) \simeq \operatorname{Im}(\alpha^*) \oplus \operatorname{Im}(\beta^*).
	\]
	Since $\partial M$ is a smooth compact manifold without boundary, both $\operatorname{Im}(i^*)$ and $\operatorname{Im}(\alpha^*)$ are finite-dimensional. Therefore, $\dim H^k(\overline{M})$ is finite if and only if $\dim H^k(\overline{M}, \partial M)$ is finite.
	
	Let $g$ be a scattering metric in the interior $M$ of a compact manifold with boundary $\overline{M}$ with boundary defining function $x$. Let $\mathsf{\Lambda}^k C^\infty(M)$ denote the space of smooth $k$-forms on $M$, and $\mathsf{\Lambda}^k C_c^\infty(M)$ the space of smooth $k$-forms with compact support on $M$. In a collar neighborhood $V$ of $\partial M$, any smooth $k$-form $\omega$ can be uniquely written as
	\[
	\omega|_V = \alpha + \mathrm{d}x \wedge \beta,
	\]
	where $\alpha$ and $\beta$ are forms on $M$ that are $x$-dependent but do not involve $\mathrm{d}x$. In local coordinates $(y_1, \dots, y_{n-1}, x)$ near the boundary, we may view $\alpha$ and $\beta$ as differential forms involving only $dy_j$ for $j = 1, \dots, n-1$. Clearly, if $\omega$ is a $k$-form, then $\alpha$ is a $k$-form and $\beta$ is a $(k-1)$-form. We call $\alpha$ the tangential part and $\beta$ the conormal part of $\omega$ in $V$. We refer to \cite{Melrose_SST, shapiro} for further details.
	
	As shown in \cite{Melrose_GST, shapiro}, the relative cohomology $H^\bullet(\overline{M}, \partial M)$ is naturally isomorphic to the cohomology of the compactly supported De Rham complex $(\mathsf{\Lambda}^\bullet C_c^\infty(M^\circ), \mathrm{d})$. Moreover, consider the complex $(\mathsf{\Lambda}^\bullet A(M), \mathrm{d})$, where each $\mathsf{\Lambda}^kA(M)$ is the space of the $\omega\in\mathsf{\Lambda}^kC^\infty(M)$ such that, in a neighbouhood $V\simeq (0,1)\times\partial M$ of the boundary, we can write
	\[\omega|_V = \omega_1 + \omega_2\wedge\mathrm{d}x,\]
	where we can regard the tangential and conormal parts as $\omega_1\in C^\infty((0,1),\mathsf{\Lambda}^kC^\infty(\partial M))$ and $\omega_2\in C^\infty((0,1),\mathsf{\Lambda}^{k-1}C^\infty(\partial M))$ satisfying
	\[\lim_{x\to 0}\|\omega_1(x)\|_{L^2(\partial M)} = \lim_{x\to 0}\|\mathrm{d}_{\partial M}\omega_1(x)\|_{L^2(\partial M)}=0,\]
	where $\mathrm{d}_{\partial M}$ is the exterior derivative on $\partial M$, and there exist $C,N>0$ such that
	\[\sup_{x\in(0,1)}\sup_{p\in\partial M}\|\omega_2(x,p)\|_g\leq Cx^N, \quad \sup_{x\in(0,1)}\sup_{p\in\partial M}\|\mathrm{d}_{\partial M}\omega_2(x,p)\|_g\leq Cx^N,\]
	and
	\[\sup_{x\in(0,1)}\sup_{p\in\partial M}\left\|\dfrac{\omega_2(x,p)}{\partial x}\right\|_g\leq Cx^N.\]
	
	Now, consider the complex $(\mathsf{\Lambda}^\bullet B(M), \mathrm{d})$ consisting of smooth forms on $\mathsf{\Lambda}^\bullet A(M)$ such that $\omega_1\in \dot{C}^\infty((0,1),\mathsf{\Lambda}^\bullet C^\infty(\partial M))$, that is,
	\[\lim_{x\to 0}\left\|x^{-i}\frac{\partial^j\omega_1(x)}{\partial x^j}\right\|_{\mathscr{H}^\ell(\partial M)}=0,\quad\forall i,j,\ell\in\mathbb{N}_0,\]
	where $\mathscr{H}^\ell(\partial M)$ is the usual $\ell$-th Sobolev space on $\partial M$. In other words, $(\mathsf{\Lambda}^\bullet B(M), \mathrm{d})$ is the complex of smooth forms on $M$ whose tangential part decays rapidly and whose conormal part grows at most polynomially near the boundary. Then the inclusions
	\[\mathsf{\Lambda}^\bullet C_c^\infty(M) \hookrightarrow \mathsf{\Lambda}^\bullet B(M) \hookrightarrow \mathsf{\Lambda}^\bullet A(M)\]
	induces isomorphisms in cohomology (see \cite{shapiro}). The validity of this extension to forms with prescribed growth is established in \cite[Proposition 6.13]{Melrose_APS}. Also, Shapiro \cite[Section 5.2]{shapiro} shows that every compactly supported closed 1-form differs from an exact element of $\mathsf{\Lambda}^1B(M)$ by a square-integrable harmonic 1-form.   
	
	\begin{remark}
		Using the Mayer–Vietoris sequence, one can show that the de Rham cohomology of a scattering manifold is finite-dimensional. Indeed, let $2M$ denote the double of $M$, obtained by gluing two copies of $M$ along their common boundary. Then, $2M$ can be covered by two open sets $U, V \simeq M$ with $U \cap V \simeq (0,1) \times \partial M$, which gives us an exact sequence
		\[
		\cdots \to H^k(2M) \to H^k(M) \oplus H^k(M) \to H^k(\partial M) \to H^{k+1}(2M) \to \cdots
		\]
		Observe that $H^k((0,1) \times \partial M) \simeq H^k(\partial M)$ and that both $H^k(\partial M)$ and $H^k(2M)$ are finite-dimensional, as $\partial M$ and $2M$ are compact manifolds without boundary. Therefore, $\dim H^k(M) < \infty$, and the same holds for the relative cohomology groups.
	\end{remark}

\subsection{Hodge Theory on Asymptotically Euclidean Manifolds} \

	Let $\star$ denote the Hodge star operator associated with the metric $g$. We define $\mathsf{\Lambda}^k L^2(M)$ as the space of (not necessarily smooth) sections $\alpha \in \mathsf{\Lambda}^k T^*M$ such that
	\[
	\int_M \alpha \wedge (\star \alpha) < +\infty,
	\]
	equipped with the inner product
	\[
	\langle \alpha, \beta \rangle = \int_M \alpha \wedge (\star \beta),
	\]
	which turns $\mathsf{\Lambda}^k L^2(M)$ into a Hilbert space. The exterior derivative can be extended to this space in the same way as weak derivatives are defined for $L^2$ functions.

	For each $k \in \mathbb{N}_0$, the Laplacian $\Delta$ on $k$-forms is defined as the Laplacian associated to the De Rham complex
	\[0\to C^\infty(M) \overset{\mathrm{d}}{\longrightarrow} \mathsf{\Lambda}^1C^\infty(M) \overset{\mathrm{d}}{\longrightarrow} \cdots \overset{\mathrm{d}}{\longrightarrow} \mathsf{\Lambda}^nC^\infty(M)\to 0,\]
	as in Subsection \ref{sect_ellcomp}. The space of square-integrable harmonic $k$-forms is defined as
	\[
	\mathcal{H}^k(M) = \{u\in\mathsf{\Lambda}^kL^2(M)\,:\, \Delta u=0\}.
	\]
		
	If $M$ is a smooth closed manifold endowed with a Riemannian metric $g$, the Hodge Theorem establishes an isomorphism between the space $\mathcal{H}^k(M)$, and the de Rham cohomology group $H^k(M)$. In general, there is no analogue of the Hodge Theorem for non-compact manifolds. However, a version of the theorem does hold in the setting of asymptotically Euclidean manifolds:
	
	\begin{theorem}[Hodge Theorem for Asymptotically Euclidean Manifolds]
		Let $M$ be an asymptotically Euclidean manifold. Then the space $\mathcal{H}^k(M)$ of harmonic $k$-forms is naturally isomorphic to the image of the inclusion map $i: H^k(\overline{M}, \partial M) \to H^k(\overline{M})$, given by $i([\omega]) = [\omega]$.
	\end{theorem}
	
	This result is stated in \cite[Theorem 6.2]{Melrose_GST}. Here, $H^k(\overline{M})$ is the cohomology of forms that are smooth up to the boundary. According to \cite[Proposition 7.4]{Melrose_APS}, the Hodge cohomology for a scattering metric coincides with that for a $b$-metric. A $b$-metric is a Riemannian metric on the interior of a compact manifold with boundary that takes the form
	\[
	g=\frac{\mathrm{d}x^2}{x^2} + g'
	\]
	in a collar neighborhood of the boundary, where $x$ is a boundary defining function and $g'$ restricts to a metric on $\partial M$. The detailed proof of the Hodge Theorem in this setting appears in \cite[Section 6.4]{Melrose_APS}, building upon foundational results of Atiyah, Patodi, and Singer \cite{APS1975}.
	
	By elliptic regularity, every harmonic $k$-form is smooth. It is important to note that the inclusion map $i$ in the theorem above is, in general, neither injective nor surjective. From now on, we denote
	\[
	H^1_{\partial M}(M) := i(H^1(\overline{M}, \partial M)) \subset H^1(\overline{M}),
	\]
	and let $\mathsf{\Lambda}^1 C^\infty_{\partial M}(M)$ denote the space of smooth $1$-forms $\omega$ such that $[\omega] \in H^1_{\partial M}(M)$.
	
	\begin{remark}
		Notice that we can regard $H^1(\overline{M})$ as a subspace of $H^1(M)$ by restriction to the interior, since the values of a smooth differential form on $\overline{M}$ is completely determined by the values in the interior by continuity. Moreover, in view of the discussion at the end of the previous subsection, the 1-forms in $\mathsf{\Lambda}^1 C^\infty_{\partial M}(M)$ can have a certain growth at infinity.
	\end{remark}

\subsection{Partial Fourier Series} \
	
	Let $V \subset \mathbb{R}^n$ be an open subset, and let $\mathbb{T}^m = \mathbb{R}^m / 2\pi\mathbb{Z}^m$ denote the $m$-dimensional torus. For a smooth function $f \in C^\infty(V \times \mathbb{T}^m)$, the \textit{partial Fourier coefficient} of $f$ associated with $\xi \in \mathbb{Z}^m$ is defined as the function $\widehat{f}_\xi \in C^\infty(V)$ given by
	\[
	\widehat{f}_\xi(t) = \frac{1}{(2\pi)^m} \int_{\mathbb{T}^m} e^{-i \xi\cdot x} f(t, x) \, \mathrm{d}x, \quad t\in V.
	\]
	
	Similarly, for a distribution $f \in \mathscr{D}'(V \times \mathbb{T}^m)$, the partial Fourier coefficient $\widehat{f}_\xi \in \mathscr{D}'(V)$ is defined by
	\[
	\langle \widehat{f}_\xi, \phi \rangle = \langle f, \phi \otimes e^{-i \xi \cdot x} \rangle, \quad \phi \in C_c^\infty(V).
	\]
	
	We then have the following result:
	
	\begin{theorem}
		Let $V \subset \mathbb{R}^n$ be an open subset and $u \in \mathscr{D}'(V \times \mathbb{T}^m)$. Then, $u \in C^\infty(V \times \mathbb{T}^m)$ if and only if the following conditions hold:
		\begin{enumerate}
			\item $\widehat{u}_\xi \in C^\infty(V)$ for all $\xi \in \mathbb{Z}^m$;
			
			\item For every compact set $K \subset V$, every multi-index $\alpha \in \mathbb{N}_0^n$, and every $N \in \mathbb{N}_0$, there exists a constant $C_N > 0$ such that
			\[
			\sup_{t \in K} |\partial_t^\alpha \widehat{u}_\xi(t)| \leq C_N (1 + |\xi|)^{-N}, \quad \forall \xi \in \mathbb{Z}^m.
			\]
		\end{enumerate}
		
		Under these conditions, we have the representation
		\[
		u(t, x) = \sum_{\xi \in \mathbb{Z}^m} \widehat{u}_\xi(t) e^{i \xi\cdot x},
		\]
		with uniform convergence on compact subsets of $V \times \mathbb{T}^m$.
	\end{theorem}
	
	Now, let $f \in \mathscr{D}'(V \times \mathbb{T}^m)$ and let $\chi: V' \subset \mathbb{R}^n \to V$ be a diffeomorphism. Define $X = \chi \times \mathrm{id}_{\mathbb{T}^m}$ and let $\varphi = X^*f \in \mathscr{D}'(V' \times \mathbb{T}^m)$. Consider the partial Fourier coefficient $\widehat{\varphi}_\xi \in \mathscr{D}'(V')$, and define $\widehat{f}_\xi = (\chi^{-1})^*(\widehat{\varphi}_\xi) \in \mathscr{D}'(V)$. This construction is independent of the choice of parametrization; for details, see \cite{ADL2023gh}. Using a partition of unity, we can extend this construction to define the partial Fourier coefficients of smooth functions and distributions on $M \times \mathbb{T}^m$.
	
	Similarly, consider the space $\mathsf{\Lambda}^{0,1} C^\infty(M \times \mathbb{T}^m)$ of smooth 1-forms on $M \times \mathbb{T}^m$ that do not depend on the differentials of the torus variable $x$. 
	Specifically, in any coordinate system $(V, t_1, \dots, t_n)$ on $M$, a 1-form $f \in \mathsf{\Lambda}^{0,1} C^\infty(M \times \mathbb{T}^m)$ is expressed as
	\[
	f|_{V \times \mathbb{T}^m} = \sum_{j=1}^n f_j(t, x) \, \mathrm{d}t_j.
	\]
	
	In each coordinate system $(V, t_1, \dots, t_n)$ on $M$, we define the Fourier coefficient $\widehat{f}_\xi$ of the form $f \in \mathsf{\Lambda}^{0,1} C^\infty(M \times \mathbb{T}^m)$ by
	\[
	\widehat{f}_\xi = \sum_{j=1}^n (\widehat{f_j})_\xi \, \mathrm{d}t_j \in \mathsf{\Lambda}^1 C^\infty(V).
	\]

	Using a partition of unity, we then obtain the global representation
	\[
	f = \sum_{\xi \in \mathbb{Z}^m} \sum_{j=1}^n (\widehat{f_j})_\xi(t) e^{i \xi \cdot x} \, \mathrm{d}t_j,
	\]
	where the smoothness of $f$ is characterized by the decay properties of the Fourier coefficients $(\widehat{f_j})_\xi$, as described in the previous theorem.
	
	\begin{proposition}
		Let $u \in \mathscr{D}'(M \times \mathbb{T}^m)$. Then, $u = 0$ if and only if $\widehat{u}_\xi = 0$ for all $\xi \in \mathbb{Z}^m$.
	\end{proposition}
	
	\begin{proof}
		See \cite[Lemma 5.1]{ADL2023gh}.
	\end{proof}

	Now let $M$ be an $n$-dimensional asymptotically Euclidean manifold, and let $\omega$ be a closed smooth real-valued 1-form on $M$. Consider the operator
	\begin{equation}\label{L}
		\mathbb{L}: C^\infty(M \times \mathbb{T}^m) \to \mathsf{\Lambda}^{0,1} C^\infty(M \times \mathbb{T}^m), \quad \mathbb{L}u = \mathrm{d}_t u + \omega \wedge \partial_x u.
	\end{equation}
	
	We then have the following consequence.
	
	\begin{corollary}
		Let $u \in \mathscr{D}'(M \times \mathbb{T}^m)$ satisfy $\mathbb{L}^0 u \in \mathsf{\Lambda}^{0,1} C^\infty(M \times \mathbb{T}^m)$. Then $\widehat{u}_\xi(t) \in C^\infty(M)$ for all $\xi \in \mathbb{Z}^m$.
	\end{corollary}
	\begin{proof}
		Suppose that $u \in \mathscr{D}'(M \times \mathbb{T}^m)$ satisfies $\mathbb{L}^0 u = f \in \mathsf{\Lambda}^{0,1} C^\infty(M \times \mathbb{T}^m)$. Taking the partial Fourier series in the toroidal variable, for each $\xi \in \mathbb{Z}^m$ we obtain
		\[
		\mathbb{L}_\xi^0 \widehat{u}_\xi(t) = \mathrm{d} \widehat{u}_\xi(t) + i (\xi \cdot \boldsymbol{\omega}) \widehat{u}_\xi(t) = \widehat{f}_\xi(t),
		\]
		where $\mathrm{d}$ denotes the exterior derivative on $M$. 
		
		Observe that $\mathbb{L}_\xi^0$ is the first operator in the complex
		\[
		0 \to C^\infty(M) \xrightarrow{\mathbb{L}_\xi^0} \mathsf{\Lambda}^1 C^\infty(M) \xrightarrow{\mathbb{L}_\xi^1} \cdots \xrightarrow{\mathbb{L}_\xi^{n-1}} \mathsf{\Lambda}^n C^\infty(M) \to 0,
		\]
		which is a zero-order perturbation of the de Rham complex. Since the principal symbols of both complexes coincide and the de Rham complex is elliptic, it follows that the perturbed complex is also elliptic. Hence, by Proposition \ref{gh_firstlevel}, the operator $\mathbb{L}_\xi^0$ is globally hypoelliptic, which implies that $\widehat{u}_\xi(t) \in C^\infty(M)$ for all $\xi \in \mathbb{Z}^m$.
	\end{proof}

	\begin{lemma}\label{lemma_BCM}
		Let $u \in \mathscr{D}'(M \times \mathbb{T}^m)$ such that $\mathbb{L}u = f \in \mathsf{\Lambda}^{0,1}C^\infty(M \times \mathbb{T}^m)$. Then, given a local chart $V \subset M$, we have $u \in C^\infty(V \times \mathbb{T}^m)$ if and only if, for every $N > 0$, there exists $C > 0$ such that, for all $\xi \in \mathbb{Z}^m$,
		\begin{equation}\label{base}
			\sup_{t \in V} |\widehat{u}_\xi(t)| \leq C (1 + |\xi|)^{-N}.
		\end{equation}
	\end{lemma}
	
	\begin{proof}
		The proof follows by induction on $|\alpha|$, and is analogous to \cite[Lemma 5.13]{AFJR2024}. If $u$ is smooth in $V \times \mathbb{T}^m$, then \eqref{base} clearly holds.
		
		Conversely, suppose that \eqref{base} holds. In the local chart $V$, we write
		\[
		f(t,x) = \sum_{j=1}^n f_j(t,x)\, \mathrm{d}t_j, \qquad 
		\omega_k(t) = \sum_{j=1}^{n} \omega_{kj}(t)\, \mathrm{d}t_j, \quad k = 1,\dots,m,
		\]
		with $f_j \in C^\infty(V \times \mathbb{T}^m)$ and $\omega_{kj} \in C^\infty(V)$.
		
		Applying the Fourier transform in the $\mathbb{T}^m$-variables, we obtain, for each $\xi \in \mathbb{Z}^m$,
		\[
		\mathbb{L}_\xi \widehat{u}_\xi(t) = d_t \widehat{u}_\xi(t) + i \, \xi \cdot \omega(t) \, \widehat{u}_\xi(t) = \widehat{f}_\xi(t).
		\]
		
		In coordinates, for each $j=1,\dots,n$, this reads
		\[
		\partial_{t_j} \widehat{u}_\xi(t) + i \sum_{k=1}^{m} \xi_k \omega_{kj}(t) \widehat{u}_\xi(t) = (\widehat{f}_j)_\xi(t)
		\]
		
		We proceed by induction on $|\alpha|$. The base case $|\alpha| = 0$ is just inequality \eqref{base}. 
		
		Assume now that for some multi-index $\beta$ with $|\beta| = \ell$, and for any $N \in \mathbb{N}$, there exists $C > 0$ such that
		\[
		\sup_{t \in V} |\partial_t^\beta \widehat{u}_\xi(t)| \leq C (1 + |\xi|)^{-N}.
		\]
		
		Let $\alpha$ be a multi-index with $|\alpha| = \ell + 1$, and write $\alpha = \beta + e_j$, where $e_j$ is the $j$-th canonical basis vector of $\mathbb{R}^n$. Differentiating the identity above, we obtain
		\[
		\partial_t^\alpha \widehat{u}_\xi(t) = \partial_t^\beta \left( (\widehat{f}_j)_\xi(t) - i \sum_{k=1}^{m} \xi_k \omega_{kj}(t) \widehat{u}_\xi(t) \right).
		\]
		
		Since $\widehat{f}_j \in C^\infty(V \times \mathbb{Z}^m)$, we have, for all $N \in \mathbb{N}$,
		\[
		\sup_{t \in V} |\partial_t^\beta (\widehat{f}_j)_\xi(t)| \leq C (1 + |\xi|)^{-N}.
		\]
		
		Moreover, by applying the Leibniz rule to the product $\omega_{kj}(t) \widehat{u}_\xi(t)$ and using the inductive hypothesis (since derivatives of order $\leq \ell$ of $\widehat{u}$ satisfy the desired decay), we conclude that $\partial_t^\alpha \widehat{u}_\xi(t)$ also satisfies the same type of bound. Therefore, $\widehat{u}_\xi(t)$ is rapidly decreasing in $\xi$ uniformly in $t$, along with all its derivatives, which implies $u \in C^\infty(V \times \mathbb{T}^m)$.
	\end{proof}

\section{Involutive Systems on $M\times\mathbb{T}$}

	Let $M$ be an $n$-dimensional asymptotically Euclidean manifold, and let $\omega$ be a smooth, closed, real-valued 1-form on $M$. In this section, we study the global regularity of the operator
	\begin{equation}\label{L}
		\mathbb{L}:C^\infty(M\times\mathbb{T})\to \mathsf{\Lambda}^{0,1} C^\infty(M\times\mathbb{T}),\quad \mathbb{L}u=\mathrm{d}_t u+\omega\wedge\partial_xu.
	\end{equation}
	
	\begin{definition}\label{defi_liou}
		Let $\omega\in \mathsf{\Lambda}^1 C^\infty_{\partial M}(M)$ be a real-valued closed 1-form. We say that $\omega$ is:
		\begin{itemize}
			\item[(i)] \emph{integral} if $\int_\sigma\omega\in 2\pi\mathbb{Z}$ for every smooth 1-cycle $\sigma$ in $M$;
			\item[(ii)] \emph{rational} if there exists $q\in\mathbb{Z}$ such that $q\omega$ is integral;
			\item[(iii)] \emph{Liouville} if $\omega$ is not rational and there exist a sequence of closed integral 1-forms $\{\theta_j\}$ and a sequence of integers $\{q_j\}$ with $q_j\geq 2$ and $q_j\to\infty$, such that $\{q_j^j(\omega-\theta_j/q_j)\}$ is bounded in $\mathsf{\Lambda}^1 C^\infty(M)$ with respect to the standard Fréchet topology.
		\end{itemize}
	\end{definition}
	
	Note that this definition depends only on the homology classes of the 1-cycles and the cohomology class of the 1-forms. Indeed, by de Rham’s Theorem, we identify $H^1(M)$ with the dual space of $H_1(M;\mathbb{R})$, which is finite-dimensional and spanned by the free part of $H_1(M;\mathbb{Z})$. The isomorphism is given by
	\[
	[\omega]\mapsto \left([\sigma]\mapsto\dfrac{1}{2\pi}\int_{\sigma}\omega\right).
	\]
	
	Since $H^1(M)$ is finite-dimensional, denote $d'=\dim H^1(M)$. Let $\mathcal{B}'=\{[\omega_1],\dots,[\omega_d]\}$ be a basis of $H^1_{\partial M}(M)$, with $d=\dim H^1_{\partial M}(M)$, where each $\omega_k$ is harmonic (by Hodge Theorem) and orthonormal with respect to the $L^2$ inner product. 
	
	Extend $\mathcal{B}'$ to a basis $\mathcal{B}'_0=\{[\omega_1],\dots,[\omega_d],[\omega_{d+1}],\dots,[\omega_{d'}]\}$ of $H^1(M)$, and let $\sigma_1,\dots,\sigma_{d'}$ be 1-cycles generating the free part $H_1(M;\mathbb{R})$ of $H_1(M;\mathbb{Z})$, chosen so that $\mathcal{B}_0=\{[\sigma_1],\dots,[\sigma_{d'}]\}$ is dual to $\mathcal{B}_0'$. Then, for any $\omega\in \mathsf{\Lambda}^1 C^\infty_{\partial M}(M)$, it holds that 
	\[ \int_{\sigma_\ell}\omega=0, \quad \text{ for } d<\ell\leq d'.\]
	
	Define the map
	\[
	I:H^1_{\partial M}(M)\to \mathbb{R}^d,\quad [\omega] \mapsto \dfrac{1}{2\pi}\left(\int_{\sigma_1}\omega,\dots,\int_{\sigma_d}\omega\right).
	\]
	
	\begin{theorem}\label{thm_liou}
		Let $\omega\in \mathsf{\Lambda}^1 C^\infty_{\partial M}(M)$ be real and closed. Then:
		\begin{itemize}
			\item[(i)] $\omega$ is integral if and only if $I([\omega])\in \mathbb{Z}^d$;
			\item[(ii)] $\omega$ is rational if and only if $I([\omega])\in \mathbb{Q}^d$;
			\item[(iii)] $\omega$ is Liouville if and only if there exist sequences $\{p_j\}\subset\mathbb{Z}^d$, $\{q_j\}\subset\mathbb{N}$ with $q_j\geq 2$ and $q_j\to \infty$, and a constant $C>0$, such that for all $j\in\mathbb{N}$ we have
			\begin{equation}\label{liou}
				\left|I([\omega])-\dfrac{p_j}{q_j}\right|\leq\dfrac{C}{q_j^j}.
			\end{equation}
		\end{itemize}
	\end{theorem}
	
	\begin{proof}
		Items (i) and (ii) follow directly from the fact that $\int_{\sigma_\ell}\omega=0$ for all $\ell=d+1,\dots, d'$.
		
		Assume $\omega$ is Liouville, and let $\{\theta_j\}$ and $\{q_j\}$ be as in Definition \ref{defi_liou}. Then,
		\[
		\left|q_j^j I\left(\left[\omega-\frac{\theta_j}{q_j}\right]\right)\right| = \sum_{\ell=1}^d\left|\dfrac{1}{2\pi}\int_{\sigma_\ell}q_j^j\left(\omega-\dfrac{\theta_j}{q_j}\right)\right|\leq C,
		\]
		due to the boundedness of the sequence $\{q_j^j(\omega-\theta_j/q_j)\}$ in the Fréchet topology. Using the $1$-norm on $\mathbb{R}^d$ and defining $p_j=I([\theta_j])\in\mathbb{Z}^d$, we obtain the desired inequality.
		
		Conversely, suppose \eqref{liou} holds. By the Hodge theorem, $I$ is an isomorphism between $\mathcal{H}^1(M)$ and $\mathbb{R}^d$. For each $j$, let $\vartheta_j\in\mathcal{H}^1(M)$ satisfy $I([\vartheta_j])=p_j$. Then each $\vartheta_j$ is integral. Let $\vartheta\in\mathcal{H}^1(M)$ with $[\omega]=[\vartheta]$, i.e., $\omega-\vartheta=\mathrm{d}\eta$ for some $\eta\in C^\infty(M)$. Then,
		\[
		\left|I\left(\left[q_j^j\left(\omega - \frac{\vartheta_j}{q_j}\right)\right]\right)\right| = \left|I\left(q_j^j\left[\vartheta - \frac{\vartheta_j}{q_j}\right]\right)\right| = q_j^j\left|I([\omega]) - \dfrac{p_j}{q_j}\right| \leq C.
		\]
		
		Since $I:\mathcal{H}^1(M)\to\mathbb{R}^d$ is an isomorphism of Fréchet spaces, it is continuous and preserves bounded sets. Hence, $\{q_j^j(\vartheta - \vartheta_j/q_j)\}$ is bounded. Because the map $\alpha \mapsto \alpha + \mathrm{d}\eta$ is continuous on $\mathsf{\Lambda}^1 C^\infty(M)$, the sequence $\{q_j^j(\omega - \vartheta_j/q_j)\}$ is also bounded. We conclude that $\omega$ is Liouville.
	\end{proof}

	\begin{proposition}\label{uni_cov}
		Let $\Pi:\hat M\to M$ be the universal covering of $M$ and $\omega$ a real and closed 1-form on $M$. The following statements are equivalent:
		\begin{enumerate}
			\item $\omega$ is integral;
			\item For every $\psi\in C^\infty(\hat{M})$ such that $\mathrm{d}\psi=\Pi^*\omega$, and for all $P,Q\in\hat{M}$ with $\Pi(P)=\Pi(Q)$, we have $\psi(P)-\psi(Q)\in 2\pi\mathbb{Z}$.
		\end{enumerate}
	\end{proposition}
	
	\begin{proof}
		Suppose first that $\omega$ is integral. Let $\psi\in C^\infty(\hat M)$ be such that $\mathrm{d}\psi=\Pi^*\omega$, and let $P,Q\in\hat M$ satisfy $\Pi(P)=\Pi(Q)$. Consider a smooth path $\gamma$ from $P$ to $Q$ in $\hat M$. Then, $\Pi\circ\gamma$ is a closed path in $M$, and thus defines a smooth 1-cycle. By Hurewicz’s Theorem \cite[Theorem 13.14]{Lee}, there exists a smooth 1-cycle $\sigma$ in $M$ homologous to $\Pi(\gamma)$. Applying Stokes' Theorem, we obtain
		\[
		\psi(P)-\psi(Q)=\int_\gamma \mathrm{d}\psi = \int_\gamma \Pi^*\omega = \int_{\Pi(\gamma)} \omega = \int_\sigma \omega \in 2\pi\mathbb{Z}.
		\]
		
		Conversely, suppose that condition (2) holds. Let $\sigma$ be a smooth 1-cycle in $M$ and let $\hat\sigma:[0,1]\to\hat M$ be a lift of $\sigma$. Denote $Q=\hat\sigma(0)$ and $P=\hat\sigma(1)$. Clearly, $\Pi(P)=\Pi(Q)$. Let $\psi\in C^\infty(\hat M)$ satisfy $\mathrm{d}\psi=\Pi^*\omega$. Then, once again by Stokes' Theorem, we have
		\[
		\int_\sigma\omega = \int_{\hat\sigma}\Pi^*\omega = \int_{\hat\sigma}\mathrm{d}\psi = \psi(P)-\psi(Q)\in 2\pi\mathbb{Z},
		\]
		which implies that $\omega$ is integral.
	\end{proof}
	
	The next theorem provides necessary and sufficient conditions for the global hypoellipticity of the operator $\mathbb{L}$ given by \eqref{L}, under the assumption that the 1-form $\omega$ is real, closed, and belongs to $\mathsf{\Lambda}^1 C^\infty_{\partial M}(M)$. This result is an adapted version of a theorem by Bergamasco, Cordaro, and Malagutti \cite[Theorem 2.4]{BCM1993}.	
	
	\begin{theorem}\label{thm_MT}
		Let $\omega \in \mathsf{\Lambda}^1 C^\infty_{\partial M}(M)$ be a real-valued closed 1-form. Then, the operator 
			\begin{equation*}
			\mathbb{L}:C^\infty(M\times\mathbb{T})\to \mathsf{\Lambda}^{0,1} C^\infty(M\times\mathbb{T}),\quad \mathbb{L}u=\mathrm{d}_t u+\omega\wedge\partial_xu
			\end{equation*}
		is globally hypoelliptic if and only if $\omega$ is neither rational nor Liouville.
	\end{theorem}
	
	\begin{proof}
		First, suppose that $\omega$ is rational and let $q \in \mathbb{N}$ be such that $q\omega$ is integral. Let $\psi \in C^\infty(\hat M)$ satisfy $\mathrm{d}\psi = \Pi^*(q\omega)$. By Proposition~\ref{uni_cov}, we may regard $e^{i\xi\psi}$ as a smooth function on $M$ for every $\xi \in \mathbb{Z}$, since $\psi(P) - \psi(Q) \in 2\pi\mathbb{Z}$ whenever $\Pi(P) = \Pi(Q)$. Consider the function
		\[
		u(t,x) = \sum_{\xi \in \mathbb{N}} e^{i\xi\psi(t)} e^{-i\xi qx}.
		\]
		Since the coefficients $e^{i\xi\psi}$ are bounded but do not decay as $|\xi| \to \infty$, we have $u \in \mathscr{D}'(M \times \mathbb{T}) \setminus C^\infty(M \times \mathbb{T})$. Moreover, it is straightforward to verify that $\mathbb{L}u = 0$, so that $\mathbb{L}$ is not globally hypoelliptic.
		
		Now suppose that $\omega$ is Liouville. Let $\{\theta_j\}$ and $\{q_j\}$ be as in Definition~\ref{defi_liou}. For each $j \in \mathbb{N}$, let $\psi_j \in C^\infty(\hat M)$ be such that $\mathrm{d}\psi_j = \Pi^*\theta_j$. Once again, by Proposition~\ref{uni_cov}, we have $e^{i\psi_j} \in C^\infty(M)$ for all $j \in \mathbb{N}$. Define
		\[
		u(t,x) = \sum_{j \in \mathbb{N}} e^{i\psi_j(t)} e^{-iq_jx}.
		\]
		As before, it is clear that $u \in \mathscr{D}'(M \times \mathbb{T}) \setminus C^\infty(M \times \mathbb{T})$. Moreover, we compute
		\begin{equation}\label{Lu_liou}
			\mathbb{L}u(t,x) = -\sum_{j \in \mathbb{N}} iq_j \left(\omega - \dfrac{\theta_j}{q_j}\right) e^{i\psi_j(t)} e^{-iq_jx}.
		\end{equation}
		
		In order to prove that $\mathbb{L}$ is not globally hypoelliptic, it suffices to show that the series in \eqref{Lu_liou} converges in $C^\infty(M \times \mathbb{T})$. Let $(V,t)$ be a local chart on $M$, and let $\hat V \subset \hat M$ be an open set such that $\Pi:\hat V \to V$ is a diffeomorphism. Since the sequence $\{q_j^j (\omega - \theta_j/q_j)\}$ is bounded in $\mathsf{\Lambda}^1 C^\infty(M)$, for each multi-index $\alpha \in \mathbb{N}_0^n$ and each compact set $K \subset V$, there exists a constant $C > 0$ such that, for all $j \in \mathbb{N}$,
		\[
		\sup_{t \in K} \left| \partial_t^\alpha \left(\omega - \dfrac{\theta_j}{q_j} \right) \right| \leq \dfrac{C}{q_j^j}.
		\]

		On the other hand, by the Faà di Bruno's Formula, for all $\beta\in\mathbb{N}_0^n$ and $\ell\in\mathbb{N}_0$, we have
		\begin{align*}
			\sup_{K\times\mathbb{T}}|\partial_x^\ell\partial_t^\beta e^{-i(q_jx-\psi_j(t))}| & \leq q_j^\ell\sup_{t\in K}|\partial_t^\beta e^{i\psi_j(t)}|\\
			& \leq q_j^\ell \sup_{K\times\mathbb{T}}\left|e^{i\psi_j(t)}\sum\prod \partial_t^{\nu}(i\psi_j\circ\Pi^{-1}(t))\right|\\
			& \leq C'q_j^\ell\left(1+\sum_{0<\nu\leq\beta}\sup_{t\in K}|\partial_t^{\nu}(\psi_j\circ\Pi^{-1})|\right),
		\end{align*}
		for some $C'>0$, where the sums in the second line are taken over partitions of the set $\{1,\dots,n\}$ and the products are taken over multi-indices satisfying certain conditions for each partition (see \cite{Hardy2006} for details). Now, notice that
		\[
		\theta_j=\mathrm{d}(\psi_j\circ\Pi^{-1})=\sum_{k=1}^n \partial_{t_k}(\psi_j\circ\Pi^{-1})\mathrm{d}t_k
		\]
		in $V$ and $\{\theta_j/q_j\}$ is bounded in $\mathsf{\Lambda}^1 C^\infty(M)$, which gives us
		\[
		1+\sum_{0<\nu\leq\beta}\sup_{t\in K}|\partial_t^{\nu}(\psi_j\circ\Pi^{-1})|\leq \tilde Cq_j,
		\]
		for some constant $\tilde C>0$. Then, there exists $C''>0$ such that, for all $j\in\mathbb{N}$,
		\[
		\sup_{K\times\mathbb{T}}|\partial_x^\ell\partial_t^{\beta}e^{-i(q_jx-\psi_j)}|\leq C''q_j^{\ell+1},
		\]
		which implies the smoothness of $\mathbb{L}u$. Therefore, $\mathbb{L}$ is not globally hypoelliptic.
		
		Finally, let us show that, if $\omega$ is neither rational nor Liouville, then $\mathbb{L}$ is globally hypoelliptic. Consider $u\in\mathscr{D}'(M\times\mathbb{T})$ and suppose that $\mathbb{L}u=f\in \mathsf{\Lambda}^{0,1} C^\infty(M\times\mathbb{T})$. Taking the partial Fourier series of $u$ and $f$, for each $\xi\in\mathbb{Z}$, we have
		\begin{equation*}\label{dteu_ef} 
			\mathrm{d}_t\widehat{u}_\xi + i\xi\omega\widehat{u}_\xi = \widehat{f}_\xi.
		\end{equation*}	
		
		Let $\psi\in C^\infty(\hat M)$ satisfy $\mathrm{d}\psi=\Pi^*\omega$, let $t_0\in M$, and choose a local chart $(V,t)$ at $t_0$ such that $V$ is diffeomorphic to an open ball in $\mathbb{R}^n$ and small enough so that there exists an open subset $\hat V\subset\hat M$ such that $\Pi:\hat V\to V$ is a diffeomorphism. Then, on $V$ we have
		\begin{equation}\label{deuef}
			\mathrm{d}_t(e^{i\xi\phi}\widehat{u}_\xi)=e^{i\xi\phi}\widehat{f}_\xi,
		\end{equation}
		for all $\xi\in\mathbb{Z}$, where $\phi=\psi\circ\Pi^{-1}$. Given $t\in V$, integrating \eqref{deuef} over the path $[t_0,t]$, by Stokes' Theorem, we have
		\[
		\int_{t_0}^t e^{i\xi\phi}\widehat{f}_\xi = e^{i\xi\phi(t)}\widehat{u}_\xi(t) = e^{i\xi\phi(t_0)}\widehat{u}_\xi(t_0),
		\]
		which gives us
		\[
		\widehat{u}_\xi(t) = e^{i\xi(\phi(t_0)-\phi(t))}\widehat{u}_\xi(t_0)+e^{-i\xi\phi(t)}\int_{t_0}^t e^{i\xi\phi}\widehat{f}_\xi.
		\]
		
		Since $|e^{i\xi(\phi(t_0)-\phi(t))}|=1$, by the smoothness of $f$ and Lemma \ref{lemma_BCM}, it is enough to show that the sequence $\{\widehat{u}_\xi(t_0)\}$ is rapidly decreasing.
		
		Let $\sigma_1,\dots,\sigma_d$ be the 1-cycles whose homological classes generate the subspace of $H_1(M;\mathbb{R})$ dual to $H^1_{\partial M}(M)$ as shown at the beginning of this section. By the Hurewicz Theorem, we may assume that the paths $\sigma_\ell$, $\ell=1,\dots,d$, are smooth loops based at $t_0$. Then, we lift them to paths $\hat\sigma_\ell:[0,1]\to\hat M$, $\ell=1,\dots,d$, with the same initial point $Q_0\in\hat M$, that is, $\hat\sigma_\ell(0)=Q_0$, for all $\ell=1,\dots,d$. Also, let us denote $Q_\ell=\hat\sigma_\ell(1)$, which of course satisfy $\Pi(Q_\ell)=t_0$, $\ell=1,\dots,d$. 
		
		Notice that, since $\mathrm{d}_t\Pi^*=\Pi^*\mathrm{d}_t$, \eqref{deuef} gives us, for all $\xi\in\mathbb{Z}$,
		\begin{equation}\label{deuef_pull}
			\mathrm{d}(e^{i\xi\psi}\Pi^*\widehat{u}_\xi)=e^{i\xi\psi}\Pi^*\widehat{f}_\xi.
		\end{equation}
		
		Integrating \eqref{deuef_pull} over each $\hat\sigma_\ell$, by Stokes' Theorem we obtain
		\begin{align*}
			\int_{Q_0}^{Q_\ell}e^{i\xi\psi}\Pi^*\widehat{f}_\xi & = e^{i\xi\psi(Q_\ell)}\Pi^*\widehat{u}_\xi(Q_\ell)-e^{i\xi\psi(Q_0)}\Pi^*\widehat{u}_\xi(Q_0)\\
			& = (e^{i\xi\psi(Q_\ell)}-e^{i\xi\psi(Q_0)})\widehat{u}_\xi(t_0).
		\end{align*}
		
		Now, if $\xi\in\mathbb{Z}\setminus\{0\}$ and $\ell=1,\dots,d$ are such that $e^{i\xi\psi(Q_\ell)}-e^{i\xi\psi(Q_0)}\neq 0$, the expression above gives us
		\begin{align}\label{u_t0}
			\widehat{u}_\xi(t_0)=\dfrac{-e^{-i\xi\psi(Q_0)}}{e^{i\xi(\psi(Q_0)-\psi(Q_\ell))}-1}\int_{Q_0}^{Q_\ell} e^{i\xi\psi}\Pi^*\widehat{f}_\xi.
		\end{align}
		
		Since $\omega$ is not rational, for each $\xi\neq 0$ there exists $\ell\in\{1,\dots,d\}$ such that $e^{i\xi\psi(Q_\ell)}-e^{i\xi\psi(Q_0)}\neq 0$. Also, by the smoothness of $f$, in order to show that $\{\widehat{u}_\xi(t_0)\}$ is rapidly decreasing, it is enough to see that there exist $C>0$ and $N\in\mathbb{N}_0$ such that, for all $\xi\in\mathbb{Z}\setminus\{0\}$ and $\ell=1,\dots,d$,
		\begin{align}\label{rapid_decay}
			|e^{i\xi(\psi(Q_0)-\psi(Q_\ell))}-1|\geq C|\xi|^{-N}.
		\end{align}
		
		Since $\omega$ is not Liouville, by Theorem \ref{thm_liou}, there exist $C,L>0$ such that, for all $p\in\mathbb{Z}^d$ and $q\in\mathbb{N}$, it holds
		\begin{align}\label{est_liou}
			\left|I([\omega])-\dfrac{p}{q}\right|\geq \dfrac{C}{q^L}.
		\end{align}
		
		Moreover, for each $\xi\in\mathbb{Z}\setminus\{0\}$, as in \cite[Theorem 2.4]{BCM1993}, one of the following conditions holds true:
		\begin{enumerate}
			\item[(a)] There exist $\ell\in\{1,\dots,d\}$ and $C>0$ such that
			\[|e^{i\xi(\psi(Q_0)-\psi(Q_\ell))}-1|\geq C;\]
			\item[(b)] For each $\ell\in\{1,\dots,d\}$, there exists $p_\ell\in\mathbb{Z}$ such that
			\[|e^{i\xi(\psi(Q_0)-\psi(Q_\ell))}-e^{i2\pi p_\ell}|\geq \dfrac{1}{2}|\xi(\psi(Q_\ell)-\psi(Q_0))-2\pi p_\ell|.\]
		\end{enumerate}
		
		If (a) holds, then \eqref{rapid_decay} holds for $N=0$. If (b) holds, we have by \eqref{est_liou} that
		\begin{align*}
			\max_{1\leq\ell\leq d}|e^{i\xi(\psi(Q_0)-\psi(Q_\ell))}-1| 
			& = \max_{1\leq\ell\leq d}|e^{-i2\pi p_\ell}||e^{i\xi(\psi(Q_0)-\psi(Q_\ell))}-e^{i2\pi p_\ell}|\\
			& \geq \frac{1}{2}\max_{1\leq\ell\leq d}|\xi(\psi(Q_0)-\psi(Q_\ell))-2\pi p_\ell|\\
			& \geq C'|\xi|\max_{1\leq\ell\leq d}\left|\dfrac{1}{2\pi}\int_{\sigma_\ell}\omega - \dfrac{p_\ell}{|\xi|}\right|\\
			& \geq C''|\xi|\left|I([\omega]) - \dfrac{p}{|\xi|}\right| \geq C|\xi|^{-L+1},
		\end{align*}
		where $p=(p_1,\dots,p_d)\in\mathbb{Z}^d$ and $\xi\neq 0$.
		
		Therefore, the denominator in \eqref{u_t0} is bounded below by a polynomial rate in $|\xi|$, while the numerator decays faster than any polynomial due to the smoothness of the data $f$. Indeed, observe first that
		\[
		\int_{Q_0}^{Q_\ell} e^{i\xi\psi} \Pi^* \widehat{f}_\xi 
		= \int_0^{2\pi} e^{i\xi(\psi \circ \hat\sigma_\ell)} \hat\sigma_\ell^* \Pi^* \widehat{f}_\xi 
		= \int_0^{2\pi} \sigma_\ell^* \widehat{f}_\xi.
		\]
		
		Since $\sigma_\ell$ is compact, let $\{U_i\}_{i \in I}$ be a finite collection of coordinate charts covering its image, and choose compact sets $K_i \subset \subset U_i$ whose interiors still cover $\sigma_\ell([0,2\pi])$. Consider a partition
		\[
		0 = \tau_0 < \tau_1 < \cdots < \tau_L = 2\pi
		\]
		such that, for each $r = 1, \dots, L$, the image $\sigma_\ell([\tau_{r-1}, \tau_r])$ is entirely contained in the interior of some $K_i$, with $i$ depending on $r$. Let $(t_1, \dots, t_n)$ be a system of local coordinates on $K_i$; then, over the interval $[\tau_{r-1}, \tau_r] \subset \mathbb{R}$, we can write
		\[
		\sigma_\ell^* \widehat{f}_\xi = \sum_{j=1}^n ((\widehat{f_j})_\xi \circ \sigma_\ell)\, \mathrm{d}(t_j \circ \sigma_\ell) = \sum_{j=1}^n ((\widehat{f_j})_\xi \circ \sigma_\ell)\, g_j\, \mathrm{d}\tau,
		\]
		for some continuous functions $g_j$ defined on $[\tau_{r-1}, \tau_r]$. Hence,
		\[
		\int_{\tau_{r-1}}^{\tau_r} e^{i\xi(\psi \circ \hat\sigma_\ell)} \sigma_\ell^* \widehat{f}_\xi 
		= \sum_{j=1}^n \int_{\tau_{r-1}}^{\tau_r} e^{i\xi(\psi \circ \hat\sigma_\ell)}\, (\widehat{f_j})_\xi(\sigma_\ell(\tau))\, g_j(\tau)\, \mathrm{d}\tau.
		\]
		
		Since $f$ is smooth, given any $N \in \mathbb{N}_0$, there exists a constant $C > 0$ independent of $r$ such that, for all $\xi \in \mathbb{Z}$,
		\[
		\sup_{1 \leq j \leq n} \sup_{t \in K_i} \left|(\widehat{f_j})_\xi(t)\right| \leq C (1 + |\xi|)^{-N}.
		\]
		
		Therefore,
		\[
		\left| \int_{\tau_{r-1}}^{\tau_r} e^{i\xi(\psi \circ \hat\sigma_\ell)} \sigma_\ell^* \widehat{f}_\xi \right| 
		\leq \sum_{j=1}^n \int_{\tau_{r-1}}^{\tau_r} \left| (\widehat{f_j})_\xi(\sigma_\ell(\tau)) \right| \left| g_j(\tau) \right| \mathrm{d}\tau 
		\leq C' (1 + |\xi|)^{-N},
		\]
		for some constant $C' > 0$ independent of $\xi \in \mathbb{Z}$ and $r \in \{1, \dots, L\}$. Consequently,
		\[
		\left| \int_{Q_0}^{Q_\ell} e^{i\xi\psi} \Pi^* \widehat{f}_\xi \right| 
		\leq \sum_{r=1}^L \left| \int_{\tau_{r-1}}^{\tau_r} e^{i\xi(\psi \circ \hat\sigma_\ell)} \sigma_\ell^* \widehat{f}_\xi \right| 
		\leq C'' (1 + |\xi|)^{-N},
		\]
		where $C'' = C' L > 0$. Hence, the sequence $\{ \widehat{u}_\xi(t_0) \}_{\xi \in \mathbb{Z}}$ decays faster than any polynomial, which implies that $u$ is smooth. This proves that $\mathbb{L}$ is globally hypoelliptic.
	\end{proof}

	\section{Involutive Systems on $M\times\mathbb{T}^m$, $m>1$}
	
	Let now $\omega_1,\dots,\omega_m\in \mathsf{\Lambda}^1 C^\infty_{\partial M}(M)$ be a family of real-valued closed 1-forms on $M$. In this section, inspired by the results of \cite{ADL2023gh}, we provide a characterization of the global regularity of the operator
	\begin{equation}\label{L_m}
		\mathbb{L}:\mathsf{\Lambda}^{0,1} C^\infty(M\times\mathbb{T}^m)\to C^\infty(M\times\mathbb{T}^m),\quad \mathbb{L}u = \mathrm{d}_tu+\sum_{k=1}^{m}\omega_k\wedge\partial_{x_k}u.
	\end{equation}
	
	\begin{definition}\label{defi_liou_Tm}
		Let us denote $\boldsymbol{\omega}=(\omega_1,\dots,\omega_m)$. We say that the family $\boldsymbol{\omega}$ is:
		\begin{itemize}
			\item[(i)] \emph{rational} if there exists $\xi\in\mathbb{Z}^m\setminus\{0\}$ such that
			\[\xi\cdot\boldsymbol{\omega} := \sum_{k=1}^m\xi_k\omega_k\]
			is an integral 1-form;
			\item[(ii)] \emph{Liouville} if $\boldsymbol{\omega}$ is not rational and there exist a sequence of closed integral 1-forms $\{\theta_j\}$ and a sequence $\{\xi_j\}\subset\mathbb{Z}^m$, with $|\xi_j|\to\infty$, such that the sequence $\{|\xi_j|^j(\xi_j\cdot\boldsymbol{\omega}-\theta_j)\}$ is bounded in $\mathsf{\Lambda}^1 C^\infty(M)$.
		\end{itemize}
	\end{definition}
	
	Let $\sigma_1,\dots,\sigma_d$ be the 1-cycles on $M$ defined at the beginning of the previous section. As in \cite{ADL2023gh}, we associate to $\boldsymbol{\omega}$ a matrix of cycles $A(\boldsymbol{\omega})\in M_{d\times m}(\mathbb{R})$ with entries given by
	\[
	A(\boldsymbol{\omega})_{\ell k} = \frac{1}{2\pi}\int_{\sigma_\ell}\omega_k,\quad \ell=1,\dots,d,\ k=1,\dots,m,
	\]
	that is, for all $\xi\in\mathbb{Z}^m$,
	\[
	A(\boldsymbol{\omega})\xi = \frac{1}{2\pi}\left(\int_{\sigma_1}\xi\cdot\boldsymbol{\omega},\dots, \int_{\sigma_d}\xi\cdot\boldsymbol{\omega}\right).
	\]
	
	Note that the definition of $A(\boldsymbol{\omega})$ depends only on the cohomology classes $[\omega_1], \dots, [\omega_m]$ in $H^1_{\partial M}(M)$. Hence, the assignment $\boldsymbol{\omega}\mapsto A(\boldsymbol{\omega})$ defines a linear map
	\[
	A:H^1_{\partial M}(M)^m\to M_{d\times m}(\mathbb{R}).
	\]
	
	\begin{definition}
		We say that a matrix $\boldsymbol{A}\in M_{d\times m}(\mathbb{R})$ satisfies the Diophantine condition $(\mathrm{DC})$ if and only if there exist constants $C,\rho>0$ such that
		\begin{equation}\label{DC}
			|\eta+\boldsymbol{A}\xi|\geq C(|\eta|+|\xi|)^{-\rho},\tag{DC}
		\end{equation}
		for all $(\xi,\eta)\in\mathbb{Z}^m\times\mathbb{Z}^d\setminus\{(0,0)\}$.
	\end{definition}
	
	\begin{lemma}\label{lemma_dioph}
		If a matrix $\boldsymbol{A} \in M_{d\times m}(\mathbb{R})$ satisfies \eqref{DC}, then $\boldsymbol{A}$ also satisfies the following condition: there exist constants $N \in \mathbb{N}$ and $C > 0$ such that, for all $\xi \in \mathbb{Z}^m \setminus \{0\}$, we have
		\begin{equation}\label{DC_equiv}
			\max_{1\leq \ell \leq d} |1 - e^{2\pi i \xi \cdot a_\ell}| \geq C(1 + |\xi|)^{-N},
		\end{equation}
		where $a_\ell \in \mathbb{R}^m$ denotes the $\ell$-th row of $\boldsymbol{A}$.
	\end{lemma}
	
	\begin{proof}
		This result can be seen as a smooth version of \cite[Lemma 8.1]{DM2020}. Suppose, by contradiction, that condition \eqref{DC_equiv} does not hold. Then, there exists a sequence $\{\xi_j\}_{j\in\mathbb{N}} \subset \mathbb{Z}^m$ such that $|\xi_{j+1}| > |\xi_j|$ for all $j \in \mathbb{N}$, and
		\begin{equation}\label{n_DC}
			\max_{1 \leq \ell \leq d} |e^{2\pi i \xi_j \cdot a_\ell}| < (1 + |\xi_j|)^{-j}.
		\end{equation}
		
		For each $j \in \mathbb{N}$, define $\eta_j = (\eta_j^{(1)}, \dots, \eta_j^{(d)}) \in \mathbb{Z}^d$ by setting $\eta_j^{(\ell)} = \lfloor \xi_j \cdot a_\ell \rfloor$, i.e., the integer part of $\xi_j \cdot a_\ell$. It follows from \eqref{n_DC} and the growth condition on $\{\xi_j\}$ that $\xi_j \cdot a_\ell + \eta_j^{(\ell)} \to 0$ as $j \to \infty$. Consequently, we obtain
		\[
		\pi |\xi_j \cdot a_\ell + \eta_j^{(\ell)}| \leq |e^{2\pi i (\xi_j \cdot a_\ell + \eta_j^{(\ell)})}| < (1 + |\xi_j|)^{-j},
		\]
		which contradicts condition \eqref{DC}. Therefore, \eqref{DC_equiv} must hold.
	\end{proof}
	
	The following result is an adaptation of Proposition 4.1 from \cite{ADL2023gh}:
	
	\begin{proposition}
		The family $\boldsymbol{\omega}$ is:
		\begin{itemize}
			\item[(i)] rational if and only if $A(\boldsymbol{\omega})(\mathbb{Z}^m \setminus \{0\}) \cap \mathbb{Z}^d \neq \varnothing$;
			\item[(ii)] Liouville if and only if $\boldsymbol{\omega}$ is not rational and $A(\boldsymbol{\omega})$ does not satisfy \eqref{DC}. 
		\end{itemize}
	\end{proposition}
	
	\begin{proof}
		(i) If $\boldsymbol{\omega}$ is rational, then there exists $\xi \in \mathbb{Z}^m \setminus \{0\}$ such that $A(\boldsymbol{\omega})\xi \in \mathbb{Z}^d$.
		
		Conversely, suppose there exists $\xi \in \mathbb{Z}^m \setminus \{0\}$ such that $A(\boldsymbol{\omega})\xi \in \mathbb{Z}^d$. Then,
		\[
		\dfrac{1}{2\pi}\int_{\sigma_\ell} \xi \cdot \boldsymbol{\omega} \in \mathbb{Z},\quad \ell = 1, \dots, d,
		\]
		which implies
		\[
		\dfrac{1}{2\pi}\int_{\sigma} \xi \cdot \boldsymbol{\omega} \in \mathbb{Z},
		\]
		for every 1-cycle $\sigma$, as in the previous section.
		
		(ii) By the Hodge theorem, there exist $\vartheta_1, \dots, \vartheta_d \in \mathcal{H}^1(M)$ such that $\{[\vartheta_1], \dots, [\vartheta_d]\}$ forms a basis for $H^1_{\partial M}(M)$ dual to the paths $\sigma_1, \dots, \sigma_d$. For each $k = 1, \dots, m$, we can write
		\begin{equation}\label{decomp_omega_k}
			\omega_k = \sum_{\ell=1}^{d} \lambda_{\ell k} \vartheta_\ell + \mathrm{d}v_k,
		\end{equation}
		for some $v_k \in C^\infty(M)$. Applying $\int_{\sigma_\ell}$ to both sides of \eqref{decomp_omega_k}, for $\ell = 1, \dots, d$, yields
		\[
		\lambda_{\ell k} = A(\boldsymbol{\omega})_{\ell k}, \quad k = 1, \dots, m,\ \ell = 1, \dots, d.
		\]
		
		For each $\xi \in \mathbb{Z}^m$, we obtain
		\begin{equation}\label{omega}
			\xi \cdot \boldsymbol{\omega} = \sum_{\ell=1}^{d} \left( \sum_{k=1}^{m} A(\boldsymbol{\omega})_{\ell k} \xi_k \right) \vartheta_\ell + \sum_{k=1}^{m} \xi_k \mathrm{d}v_k.
		\end{equation}
		
		Suppose $A(\boldsymbol{\omega})$ does not satisfy condition \eqref{DC}. Then, there exists a sequence $\{(\eta_j, \xi_j)\}_{j \in \mathbb{N}}$  in $\mathbb{Z}^d \times (\mathbb{Z}^m \setminus \{0\})$ such that, for all $j \in \mathbb{N}$,
		\begin{equation}\label{not_DC}
			|\eta_j + A(\boldsymbol{\omega})\xi_j| < (|\eta_j| + |\xi_j|)^{-j}.
		\end{equation}
		
		We first show that the sequence $\{\xi_j\}$ cannot be bounded. Indeed, if it were, then
		\[
		|\eta_j + A(\boldsymbol{\omega})\xi_j| \geq |\eta_j| - C,
		\]
		for some $C > 0$, which would imply that $\{\eta_j\}$ is also bounded, contradicting \eqref{not_DC}.
		
		Now, for each $j \in \mathbb{N}$, define the 1-form
		\begin{equation}\label{vartheta}
			\theta_j = -\sum_{\ell=1}^{d} (\eta_j)_\ell \vartheta_\ell + \sum_{k=1}^{m} (\xi_j)_k \mathrm{d}v_k.
		\end{equation}
		
		Notice that each $\theta_j$ is integral. Indeed, each component $(\eta_j)_\ell$ belongs to $\mathbb{Z}$, each $\vartheta_\ell$ was chosen to be integral, and each $\mathrm{d}v_k$ is integral by exactness. Using \eqref{omega} and \eqref{vartheta}, we have
		\begin{align*}
			\rho_j & := |\xi_j|^j(\xi_j \cdot \boldsymbol{\omega} - \theta_j)\\
			& = |\xi_j|^j\left[\sum_{\ell=1}^{d}\left(\sum_{k=1}^{m}A(\boldsymbol{\omega})_{\ell k}(\xi_j)_k\right)\vartheta_\ell + \sum_{k=1}^{m}(\xi_j)_k\,\mathrm{d}v_k + \sum_{\ell=1}^{d}(\eta_j)_\ell\vartheta_\ell - \sum_{k=1}^{m}(\xi_j)_k\,\mathrm{d}v_k\right]\\
			& = |\xi_j|^j \sum_{\ell=1}^{d} \left(\sum_{k=1}^{m} A(\boldsymbol{\omega})_{\ell k}(\xi_j)_k + (\eta_j)_\ell\right)\vartheta_\ell.
		\end{align*}
		
		To show that $\boldsymbol{\omega}$ is Liouville, it suffices to prove that the sequence $\{\rho_j\}$ is bounded in $\mathsf{\Lambda}^1 C^\infty(M)$. Let $(V,t)$ be a local chart, and let $K \subset V$ be a compact subset. Then, for each multi-index $\alpha \in \mathbb{N}_0^n$, we have
		\begin{align*}
			\sup_{t \in K} |\partial_t^\alpha \rho_j| 
			& = |\xi_j|^j \sum_{\ell=1}^{d} \left| \sum_{k=1}^{m} A(\boldsymbol{\omega})_{\ell k}(\xi_j)_k + (\eta_j)_\ell \right| \cdot |\partial_t^\alpha \vartheta_\ell| \\
			& \leq C |\xi_j|^j \left| \eta_j + A(\boldsymbol{\omega})\xi_j \right| \ <\ +\infty,
		\end{align*}
		by \eqref{not_DC} and the fact that each $\vartheta_\ell$ is smooth (since it is harmonic). Therefore, $\boldsymbol{\omega}$ is Liouville.

		Conversely, suppose that $\boldsymbol{\omega}$ is Liouville. Then there exist a sequence $\{\xi_j\}$ in $\mathbb{Z}^m$, with $|\xi_j| \to \infty$, and a sequence of closed integral 1-forms $\{\theta_j\}$ such that $\{|\xi_j|^j(\xi_j \cdot \boldsymbol{\omega} - \theta_j)\}$ is bounded in $\mathsf{\Lambda}^1 C^\infty(M)$. As before, write
		\begin{equation}\label{theta}
			\theta_j = \sum_{\ell=1}^{d} \mu_{j\ell} \vartheta_\ell + \mathrm{d}w_j,
		\end{equation}
		with $\mu_{j\ell} \in \mathbb{Z}$ and $w_j \in C^\infty(M)$. Then, combining equations \eqref{omega} and \eqref{theta}, we have
		\begin{equation}\label{rho_j}
			\rho_j = |\xi_j|^j \sum_{\ell=1}^{d} \left( \sum_{k=1}^{m} A(\boldsymbol{\omega})_{\ell k} (\xi_j)_k - \mu_{j\ell} \right) \vartheta_\ell + |\xi_j|^j \left( \sum_{k=1}^{m} (\xi_j)_k \mathrm{d}v_k - \mathrm{d}w_j \right).
		\end{equation}
		
		Integrating \eqref{rho_j} over each 1-cycle $\sigma_\ell$, for $\ell = 1, \dots, d$, we obtain
		\[
		\frac{1}{2\pi} \int_{\sigma_\ell} \rho_j = |\xi_j|^j \left( \sum_{k=1}^{m} A(\boldsymbol{\omega})_{\ell k} (\xi_j)_k - \mu_{j\ell} \right).
		\]
		
		Since $\{\rho_j\}$ is bounded in $\mathsf{\Lambda}^1 C^\infty(M)$, there exists $C > 0$ such that
		\[
		\left| \int_{\sigma_\ell} \rho_j \right| \leq C,
		\]
		for all $\ell$ and $j$. Defining $\eta_j := -(\mu_{j1}, \dots, \mu_{jd}) \in \mathbb{Z}^d$, we conclude
		\[
		|\eta_j + A(\boldsymbol{\omega}) \xi_j| \leq C' |\xi_j|^{-j},
		\]
		for some $C' > 0$. Since $|\xi_j| \to \infty$, passing to a subsequence if necessary, we may assume $|\xi_j|^{-j/2} \leq (C'j)^{-1}$. Hence,
		\[
		|\eta_j + A(\boldsymbol{\omega}) \xi_j| \leq \frac{1}{j} |\xi_j|^{-j/2}, \quad \forall j \in \mathbb{N},
		\]
		which implies that \eqref{DC} does not hold.
	\end{proof}
	
	Finally, we state an adaptation of the main result from \cite{ADL2023gh}. Some steps of the proof are omitted due to their similarity with those in the proof of Theorem~\ref{thm_MT}:
	
	\begin{theorem}
		Let $\boldsymbol{\omega}=(\omega_1, \dots, \omega_m) \in \mathsf{\Lambda}^1 C^\infty_{\partial M}(M)$ be a family of real-valued closed 1-forms on $M$. The operator 
		\begin{equation*}
			 \mathbb{L} = \mathrm{d}_t+\sum_{k=1}^{m}\omega_k\wedge\partial_{x_k}
		\end{equation*}
		is globally hypoelliptic if and only if $\boldsymbol{\omega}$ is neither rational nor Liouville.
	\end{theorem}
	
	\begin{proof}
		First, suppose that the family $\boldsymbol{\omega}$ is rational. Then, there exists $\xi_0\in\mathbb{Z}^m\setminus\{0\}$ such that $\xi_0\cdot\boldsymbol{\omega}$ is integral. Let $\Pi:\hat M\to M$ be the universal covering of $M$. By Proposition \ref{uni_cov}, there is $\psi\in C^\infty(\hat M)$ such that $\mathrm{d}\psi=\Pi^*(\xi_0\cdot\boldsymbol{\omega})$ and $e^{i\psi}$ is a smooth function on $M$. Define
		\[
		u=\sum_{j=1}^{\infty}e^{ij\psi(t)}e^{-ij\xi_0\cdot x}.
		\]
		
		It is clear that $u\in\mathscr{D}'(M\times\mathbb{T}^m)\setminus C^\infty (M\times\mathbb{T}^m)$ and $\mathbb{L}u=0$ as in the case $m=1$ (cf. Lemma \ref{lemma_BCM}). Then, $\mathbb{L}$ is not globally hypoelliptic.
		
		Now, suppose that $\boldsymbol{\omega}$ is Liouville. Let $\{\xi_j\}$ and $\{\theta_j\}$ be sequences as in Definition \ref{defi_liou_Tm}. For each $j\in\mathbb{N}$, let $\psi_j\in C^\infty(\hat M)$ be such that $\mathrm{d}\psi_j=\Pi^*(\theta_j)$. Again, by Proposition \ref{uni_cov}, each $e^{i\psi_j}$ is a smooth function on $M$. Define
		\[
		u=\sum_{j=1}^{\infty} e^{i\psi_j(t)}e^{-i\xi_j\cdot x}.
		\]
		
		Then, $u\in\mathscr{D}'(M\times\mathbb{T}^m)\setminus C^\infty (M\times\mathbb{T}^m)$. Let us show that $\mathbb{L}u$ is smooth. We have that
		\[
		\mathbb{L}u=-i\sum_{j=1}^{\infty}\left(\xi_j\cdot\boldsymbol{\omega}-\theta_j\right)e^{i\psi_j(t)}e^{-i\xi_j\cdot x}.
		\]
		
		The smoothness of $\mathbb{L}u$ follows from the fact that $\{|\xi_j|^j(\xi_j\cdot\boldsymbol{\omega}-\theta_j)\}$ is bounded. Therefore, given a small enough local chart $(V,t)$ on $M$, $K\subset V$ compact, and $\alpha\in\mathbb{N}_0^n$, there exists $C>0$ such that
		\[
		\sup_{t\in K}\left| \partial_t^\alpha\left(\xi_j\cdot\boldsymbol{\omega}-\theta_j\right) \right|\leq \dfrac{C}{|\xi_j|^j}.
		\]
		
		Arguing as in Theorem \ref{thm_MT}, we conclude that $\mathbb{L}u$ is smooth, which shows that $\mathbb{L}$ is not globally hypoelliptic.
		
		Finally, suppose that $\boldsymbol{\omega}$ is neither rational nor Liouville. Let $u\in\mathscr{D}'(M\times\mathbb{T}^m)$ be such that $\mathbb{L}u\in \mathsf{\Lambda}^1 C^\infty(M\times\mathbb{T}^m)$. Taking the Fourier series, for each $\xi\in\mathbb{Z}^m$, we obtain
		\[
		\mathrm{d}_t\widehat{u}_\xi+i(\xi\cdot\boldsymbol{\omega})\widehat{u}_\xi=\widehat{f}_\xi.
		\] 
		
		For each $k=1,\dots,m$, take $\psi_k\in C^\infty(\hat M)$ such that $\mathrm{d}\psi_k=\Pi^*(\omega_k)$. Fix $t_0\in M$ and let $(B,t)$ be a local chart on $M$ diffeomorphic to an open ball, small enough such that there exists $\hat B\subset \hat M$ for which $\Pi:\hat B\to B$ is a diffeomorphism. Then, as in Theorem \ref{thm_MT}, we have
		\[
		\mathrm{d}_t(e^{i\xi\cdot\boldsymbol{\phi}}\widehat{u}_\xi) =e^{i\xi\cdot\boldsymbol{\phi}} \widehat{f}_\xi,
		\]
	where $\boldsymbol{\phi}=(\psi_1\circ\Pi^{-1},\dots, \psi_m\circ\Pi^{-1})$. Integrating the expression above over the path $[t_0,t]$, by Stokes' Theorem, we obtain
	\[
	\widehat{u}_\xi(t)=e^{i\xi\cdot[\boldsymbol{\phi}(t_0)-\boldsymbol{\phi}(t)]}\widehat{u}_\xi(t_0) + e^{-i\xi\cdot\boldsymbol{\phi}(t_0)}\int_{t_0}^{t} e^{i\xi\cdot\boldsymbol{\phi}}\widehat{f}_\xi.
	\]
	
	In view of Lemma \ref{lemma_BCM}, due to the smoothness of $f$ and the fact that $\boldsymbol{\phi}$ is real-valued, it is enough to show that $\{\widehat{u}_\xi(t_0)\}$ is rapidly decreasing.
	
	Recall that
	\begin{equation}\label{aaaaa} 
		\mathrm{d}(e^{i\xi\cdot\boldsymbol{\psi}}\Pi^*(\widehat{u}_\xi))=e^{i\xi\cdot\boldsymbol{\psi}}\Pi^*(\widehat{f}_\xi),
	\end{equation}
	where $\boldsymbol{\psi}=(\psi_1,\dots,\psi_m)$. For each $\ell=1,\dots,d$, let $\sigma_\ell:[0,1]\to M$ be paths whose homology classes generate a subspace of $H_1(M;\mathbb{R})$ dual to $H_{\partial M}^1(M)$, as in the previous section. By Hurewicz's Theorem, we can choose the paths to be smooth and based at $t_0$. Also, let $\hat\sigma_\ell:[0,1]\to\hat M$ be liftings of each $\sigma_\ell$ such that $\hat\sigma_\ell(0)=Q_0$ for all $\ell=1,\dots,d$, for some $Q_0$. Let us denote $Q_\ell=\hat\sigma_\ell(1)$.
	
	Integrating \eqref{aaaaa} over each $\hat\sigma_\ell$, we obtain
	\[
	\widehat{u}_\xi(t_0)=e^{i\xi\cdot(\boldsymbol{\psi}(Q_\ell)-\boldsymbol{\psi}(Q_0))}\widehat{u}_\xi(t_0)+e^{-i\xi\cdot\boldsymbol{\psi}(Q_0)}\int_{Q_0}^{Q_\ell}e^{i\xi\cdot\boldsymbol{\psi}}\Pi^{*}\widehat{f}_\xi,\quad\ell=1,\dots,d.
	\]
	
	Now, notice that
	\[
	\psi_k(Q_\ell)-\psi_k(Q_0) = \int_{\hat\sigma_\ell}\mathrm{d}\psi_k = \int_{\hat\sigma_\ell}\Pi^*\omega_k = \int_{\sigma_\ell}\omega_k = 2\pi A(\boldsymbol{\omega})_{\ell k},
	\]
	for \( \ell=1,\dots,d,\) and \(\ k=1,\dots, m.\)
	
	Then,
	\[
	A(\boldsymbol{\omega})\xi = \dfrac{1}{2\pi}\big(\xi\cdot[\boldsymbol{\psi}(Q_1)-\boldsymbol{\psi}(Q_0)],\dots,\xi\cdot[\boldsymbol{\psi}(Q_d)-\boldsymbol{\psi}(Q_0)]\big).
	\]
	
	It follows from Lemma \ref{lemma_dioph} that there exist $N\in\mathbb{N}$ and $C>0$ such that
	\[
	\max_{1\leq\ell\leq d}|1-e^{i\xi\cdot[\boldsymbol{\psi}(Q_0)-\boldsymbol{\psi}(Q_\ell)]}| = \max_{1\leq\ell\leq d}|1-e^{-2\pi i(A(\boldsymbol{\omega})\xi)_\ell}| \geq C(1+|\xi|)^{-N},
	\]
	for every \( \xi\in\mathbb{Z}^m\setminus\{0\}. \)
	
	Choosing $\ell\in\{1,\dots,d\}$ such that the maximum is attained (and consequently nonzero), we have
	\[
	\widehat{u}_\xi(t_0) = \dfrac{-e^{-i\xi\cdot\boldsymbol{\psi}(Q_\ell)}}{e^{i\xi\cdot[\boldsymbol{\psi}(Q_0)-\boldsymbol{\psi}(Q_\ell)]}-1}\int_{Q_0}^{Q_\ell}e^{i\xi\cdot\boldsymbol{\psi}}\Pi^*\widehat{f}_\xi.
	\]
	
	Then,
	\[
	|\widehat{u}_\xi(t_0)| = \dfrac{1}{\left|1-e^{i\xi\cdot[\boldsymbol{\psi}(Q_0)-\boldsymbol{\psi}(Q_\ell)]}\right|}\left|\int_{Q_0}^{Q_\ell}e^{i\xi\cdot\boldsymbol{\psi}}\Pi^*\widehat{f}_\xi\right| \leq C(1+|\xi|)^{N}\left|\int_{Q_0}^{Q_\ell}e^{i\xi\cdot\boldsymbol{\psi}}\Pi^*\widehat{f}_\xi\right|.
	\]
	
	We conclude that $\{\widehat{u}_\xi(t_0)\}$ is rapidly decreasing due to the smoothness of $f$ and the fact that $\boldsymbol{\psi}$ is real, which finishes the proof.
	\end{proof}

	\section{Final Remarks: $(a,b)$-Metrics and Related Geometric Structures}

	Let $\overline{M}$ be a $n$-dimensional smooth compact manifold with boundary, and $x$ be a boundary defining function. We say that $g$ is an $(a,b)$-metric on the interior $M$ of $\overline{M}$ if $g$ is of the form
	\[
	g = \frac{\mathrm{d}x^2}{x^{2a+2}}+\frac{g'}{x^{2b}}
	\]
	near the boundary, with $a,b\geq 0$. 
	
	The scattering metric considered in the previous sections is then a $(1,1)$-metric. When $(a,b)=(0,0)$, the metric $g$ is called a $b$-metric. In this case, Melrose showed in \cite{Melrose_GST} that $\mathcal{H}^k(M)$ is isomorphic to the image of the inclusion $i:H^k(\overline{M},\partial M)\to H^k(\overline{M})$, as in the case of scattering metrics.
	
	We note that $b$-metrics arise naturally on manifolds with cylindrical ends. Indeed, under a suitable  change of variables, a $b$-metric becomes asymptotically equivalent to a product metric on $\mathbb{R} \times \partial M$, justifying the interpretation of the corresponding non-compact manifold as having cylindrical ends.

	More generally, for $(a,b)$-metrics with $a>b>0$, we have the following result due to Shapiro \cite{shapiro}:
	
	\begin{theorem}[Shapiro]\label{thm:Shapiro}
		Let $M$ be a smooth compact manifold with boundary endowed with an $(a,b)$-metric, with $a\geq b\geq 0$. 
		\begin{enumerate}
			\item If $b=0$, then $\mathcal{H}^k(M)$ is isomorphic to the image of 
		\[i:H^k(\overline{M},\partial M)\to H^k(\overline{M})\] 
		
		\item If $b>0$, then
		\[
		\mathcal{H}^k(M)\simeq \begin{cases}
				H^k(\overline{M},\partial M), &\text{if } k < (n+1 - a/b)/2,\\
				\operatorname{Im}(i:H^k(\overline{M},\partial M)\to H^k(\overline{M})), &\text{if } k = (n+1 - a/b)/2,\\
				H^k(\overline{M}), &\text{if } k > (n+1 - a/b)/2.
			\end{cases}
		\]
				\end{enumerate}
		
	\end{theorem}
	
	Another important class of manifolds is that of conformally compact metrics, which correspond to the case $(a,b) = (0,1)$. These metrics naturally arise in the study of hyperbolic spaces. For this class, Mazzeo proved the following result in \cite{mazzeo}:

\begin{theorem}[Mazzeo]\label{thm:Mazzeo}
	Let $\overline{M}$ be a smooth compact manifold with boundary endowed with a conformally compact metric. Then
	\[
	\mathcal{H}^k(M)\simeq \begin{cases}
			H^{k}(\overline{M},\partial M), &\text{if } k<(n-1)/2,\\
			H^k(\overline{M}), &\text{if } k>(n+1)/2,
		\end{cases}
	\]
	and $\mathcal{H}^k(M)$ is infinite-dimensional for $k = n/2$.
\end{theorem}

In view of Theorems \ref{thm:Shapiro} and \ref{thm:Mazzeo}, the results presented in the previous sections could be proved for a broader class of metrics 
on the interior of a compact manifold with boundary, provided that suitable versions of the Hodge theorem hold true in such situations. 
In the case of conformally compact metrics, one might assume $\dim M = n > 4$ to avoid the critical degree $k = n/2$.

\bibliographystyle{plain}
\bibliography{references}
	
\end{document}